\documentclass{article}
\usepackage{enumerate}
\usepackage{fancyhdr}
\usepackage{amsmath}
\usepackage{amsfonts}
\usepackage{amsthm}
\usepackage{color}
\usepackage{latexsym}
\usepackage{amssymb}
\usepackage{dsfont}

\usepackage{amsbsy}
\usepackage{url}
\usepackage{graphicx}
\usepackage{caption,subcaption}
\usepackage{tikz}
\usetikzlibrary{arrows}

\newtheorem{definition}{Definition}
\newtheorem{theorem}{Theorem}
\newtheorem{lemma}{Lemma}
\newtheorem{example}{Example}
\newtheorem{corollary}{Corollary}
\newtheorem{remark}{Remark}

\begin{document}
\title{\bf Exponentially stable solution of mathematical model based on graph theory of agents dynamics on time scales}

\author{Urszula Ostaszewska\footnote{University of Bialystok, Poland, email: uostasze@math.uwb.edu.pl},
Ewa Schmeidel\footnote{University of Bialystok, Poland, email: eschmeidel@math.uwb.edu.pl}, 
Ma\l gorzata Zdanowicz\footnote{University of Bialystok, Poland, email: mzdan@math.uwb.edu.pl}}

\date{\today}

\maketitle
\begin{abstract}
\noindent In this paper an emergence of leader-following model based on graph theory on the arbitrary time scales is investigated. It means that the step size is not necessarily constant but it is a function of time. We propose and prove conditions ensuring a leader-following consensus for any time scales using Gr\"{o}nwall inequality. The presented results are illustrated by examples.
\end{abstract}

\section{Introduction} 
The aim of this paper is to propose the conditions ensuring  the consensus  of multi-agent system over the arbitrary time scale.  
We consider  continuous-time and discrete-time models and also models on time scales consisting of the both kinds of points: right-dense and right-scattered simultaneously. Under some assumptions, we prove that consensus can be achieved exponentially if the graininess function is bounded. All theorems are still true if graininess function approaches zero. Some existing results of discrete-time consensus are special cases of results presented in this paper. 

The leader-following problem has been investigated since 1970s. In 1974 \cite{DG1974}, DeGroot studied   explicitly described model that resulted in the consensus. In 2000, Krause \cite{Krause2000,HK2002} discussed the model of a group of agents who have to make a joint assessment of a certain magnitude. The coordination of groups of mobile autonomous agents based on the nearest neighbor rules was considered by  Jadbabaie et al. in \cite{Jadbabaie}. Blondel et al. \cite{Blondel2009,Blondel2010}, took into account  Krause's model with state-dependent connectivity. Girejko et al. \cite{GMMM,GMSZ2016}, examined Krause's model on discrete time scales. In 2007 there were published two important articles written by Cucker and Smale \cite{CS2007_1,CS2007_2}. The authors considered an emergent behavior in flocks. Cucker-Smale model on isolated time scales was in the area of interests of Girejko et al. \cite{GMSZ2016}. Last year, Girejko, Machado, Malinowska and Martins  \cite{GMMM2018} investigated the sufficient conditions for consensus in the Cucker-Smale type model on discrete time scales. In 2015, Wang et al. \cite{Wang} published some results for  the leader-following consensus of discrete time linear multi-agent systems with communication noises.  

Presented here results generalize and improve results obtained by the authors in \cite{OSZAIP2018} and \cite{osz2019}. In \cite{osz2019} consensus on different types of discrete time scales is considered under assumption that feedback control gain $\gamma$ is constant.


The background of time scales theory is given in Bohner and Peterson books~\cite{bookBohner2001,bookBohner2003}. 

\section{Mathematical model of agents dynamics}
We consider a multi-agent system consisting of $N$ agents and the leader. The leader, labeled as $i=0$, is an isolated agent with the given trajectory $x_0 \colon \mathbb{T}\to \mathbb{R}^N$.
The dynamics of agents is described by the following equation 
\begin{equation}\label{eq5}
(x(t)-x_0(t))^\Delta(t)= F(t,x(t))-F(t,x_0(t)\mathds{1})- \gamma(t) B (x(t)-x_0(t)), 
\end{equation}
with initial condition $x(T_0)=x_{T_0}$.  Here $x \colon \mathbb{T}\to \mathbb{R}^N$ is unknown vector which represents the state of agents $x(t)=(x_1(t), \cdots, x_N(t))$ at time $t \in \mathbb{T}$, $F \colon \mathbb{T}\times\mathbb{R}^N \to \mathbb{R}^N$, $\mathds{1}$ is the vector $(1,\ldots,1)^T$, $\gamma\colon \mathbb{T} \to \mathbb{R}$ is the feedback control gain at time $t$, and $B$ is the symmetric matrix  (for details see \cite{S2019}). Notice that this model, based on the graph theory, was studied by many authors including Yu, Jiang and Hu~\cite{YJH2015}. 

Let us denote by $\varepsilon_i(t)=x_i(t)-x_0(t)$ the distance between the leader and the $i$-th agent.
If 
$(-\gamma(t) B)$ is regressive,
then by $e_{-\gamma B}(t,T_0)$ we denote a solution of initial value problem 
\[
\varepsilon^\Delta(t)=- \gamma(t) B \varepsilon(t), \,\,\, \varepsilon(T_0)=\mathds{1}.
\]
By variation of constants (see \cite{bookBohner2001}), the unique solution of equation \eqref{eq5} is given~by
\begin{equation}\label{e2}
\varepsilon(t)= e_{-\gamma B}(t, T_0)\varepsilon_{T_0}
+\int\limits_{T_0}^{t}e_{-\gamma B}(t, \sigma(\tau))\big(F(\tau,x_0(\tau)\mathds{1})-F(\tau,x(\tau))\big) \Delta \tau.
\end{equation}

\begin{definition}
Function $F$ fulfills Lipschitz  condition with respect to the second variable if there exists a positive constant $\mathcal{L}$ such that
\begin{equation}\label{Lip}
\Vert F(t,x)-F(t,y) \Vert \leq \mathcal{L} \Vert x - y \Vert, \,\,\, t  \in \mathbb{T}.
\end{equation} 
\end{definition}

\begin{definition}
We say that equation \eqref{eq5}, where $T_0 \geq 0$, $\varepsilon_{T_0}=\varepsilon(T_0) \in \mathbb{R}^N$, is exponentially stable if there exist a positive constants $c$ and $d$ such that for any rd-continuous solution $\varepsilon(t, T_0, \varepsilon_{T_0})$ of equation \eqref{eq5} holds
\[
\lim\limits_{t \to \infty}\Vert \varepsilon(t, T_0, \varepsilon_{T_0} )\Vert = \colon\lim\limits_{t \to \infty}\Vert \varepsilon(t) \Vert \leq c \Vert \varepsilon_{T_0} \Vert \lim\limits_{t \to \infty} e_d(t, T_0)  =0.
\]
\end{definition}
For some relevant result for exponential stability in the discrete case see \cite{bms2016} and \cite{Elaydi}.
In 2005 \cite{PR2005}, Peterson and Raffoul investigated the exponential stability of the zero solution
to systems of dynamic equations on time scales. The authors defined suitable Lyapunov-type
functions and then formulated appropriate inequalities on these functions that guarantee
that the zero solution decay to zero exponentially. For the growth of generalized exponential functions on time scales see Bodine and Lutz \cite{BL2003}. 

\section{Main results}
Through this paper we assume that
\[
\inf \mathbb{T} = T_0 \geq 0 \,\, \mbox{ and } \,\,\sup \mathbb{T} = \infty.
\]
It implies that $\mathbb{T}^\kappa = \mathbb{T}$. Assume that function $F$ 
satisfies Lipschitz  condition with respect to the second variable.

Let $\lambda_i$, $i=1,2,\dots,N$ denote the eigenvalues of matrix $B$. 
By $\mathbb{T}^s$ and $\mathbb{T}^d$ we mean the set of right-scattered and right-dense points of $\mathbb{T}$, respectively. Notice that, since we assumed $\sup\mathbb{T}= \infty$, at least one of sets $\mathbb{T}^s$ or $\mathbb{T}^d$ must be unbounded.

Next, we rewrite time scale $\mathbb{T}$ in the useful way for estimation of norm of solution of initial value problem \eqref{eq5} on a time scale consisting of right-scattered as well as right-dense points. To avoid confusion we underline that any interval throughout this paper is an interval on the time scale, i.e. any interval contains only points of the time scale.
Set
\[
T_1 = \min\{t \colon t \in [T_0, \infty) \cap \mathbb{T}^d \mbox{ and }  [T_0,t) \subset \mathbb{T}^s \}
\] 
\[
T_{2i} = \inf \{t \colon t \in [T_{2i-1}, \infty) \cap \mathbb{T}^s \mbox{ and } [T_{2i-1},t) \subset \mathbb{T}^d\}
\]
\[
T_{2i+1} = \min\{t \colon t \in [T_{2i}, \infty) \cap \mathbb{T}^d \mbox{ and } [T_{2i},t) \subset \mathbb{T}^s\}
\]
for $i=1,2, \dots$. In case of $[T_{2i-1}, \infty) \cap \mathbb{T}^s = \emptyset$ for some  $i \in \mathbb{N}$ we take $T_{2i}=\infty$ (see Example \ref{Ex9}). If $[T_{2i-1}, t) \cap \mathbb{T}^d = [T_{2i-1}, T_{2i-1}) =\emptyset$ for some  $i \in \mathbb{N}$ we also take $T_{2i}=\infty$ (see Example \ref{Ex2}).   Analogously, if $[T_{2i}, \infty) \cap \mathbb{T}^d = \emptyset$ for some  $i \in \mathbb{N}$, then $T_{2i+1}=\infty$. If $T_j = \infty$ for some $j \in \mathbb{N}_0$, then we take $T_{j+i}=T_j$ for $i \in \mathbb{N}$ and $[T_{j+i}, T_{j+i+1}) = (T_{j+i}, T_{j+i+1}]=\emptyset$ (see Example \ref{Ex6}). We see if $\sigma(T_0)=T_0$, then $T_1=T_0$. 
\begin{example}
Let $\mathbb{T}=  \{1\} \cup [2,3] \cup [6, \infty)$. Here $T_0=1$, $T_1=2$, $T_2=3$, $T_3=6$ and $T_4=\infty$.
\end{example}
We underline that $T_{2i+1} \in \mathbb{T}^d$ for any $i \in \mathbb{N}_0$ while it is possible $T_{2i} \notin \mathbb{T}^s$ for some $i \in \mathbb{N}_0$.
\begin{example}
Let $\mathbb{T}=  \bigcup_{i=1}^{\infty}[2i-1, 2i] \cup \{4i + \frac{1}{j+1} \colon i, j \in \mathbb{N}\}$. Here $T_0=1$, $T_1=T_0$, $T_{i}=i$ for $i \in \{2,3, \dots \}$. We see $T_1=T_0 \in \mathbb{T}^d$, $T_{2i+1} \in \mathbb{T}^d$, $T_{4i} \in \mathbb{T}^d$, $T_{4i+2} \in \mathbb{T}^s$ for $i \in  \mathbb{N}$.
\end{example}
\begin{example}
Let $\mathbb{T}=  \bigcup_{i=1}^{\infty}[2i-1, 2i] \cup \{4i+1 - \frac{1}{j+1} \colon i, j \in \mathbb{N}\}$. Here $T_{0}=1$, $T_{i}=i$ for $i \in \mathbb{N}$ and $T_{2i-1} \in \mathbb{T}^d$ and $T_{2i} \in \mathbb{T}^s$ for $i \in  \mathbb{N}$.
\end{example}

We can write 
\[
\mathbb{T}= \{T_0\}\cup \bigcup\limits_{j=0}^{\infty}(T_j, T_{j+1}]= \{T_0\}\cup\bigcup\limits_{i=0}^{\infty}(T_{2i}, T_{2i+1}] \cup (T_{2i+1}, T_{2i+2}]
\]
wherein $(T_{2i}, T_{2i+1}) \subset \mathbb{T}^s$ and $(T_{2i+1}, T_{2i+2}) \subset \mathbb{T}^d$.

In the next lemma, for any $i \in  \mathbb{N}$, the estimations of the norm of matrices $e_{-\gamma B}(t,T_{2i}) $ where $t \in [T_{2i}, T_{2i+1})$, and $e_{-\gamma B}(t,T_{2i+1})$ where $t \in [T_{2i+1}, T_{2i+2})$ are presented.
\begin{lemma}\label{L1}
If for  $i = 1,2, \dots, N$ the following conditions are satisfied
\begin{equation}\label{ewa13_61}
\gamma(t)\lambda_i \in (0, \infty) \mbox{ for }t \in \mathbb{T},
\end{equation} 
\begin{equation}\label{ewa13_7}
0< \delta \leq \mu(t) \gamma(t) \lambda_i <1 \mbox{ for } t \in \mathbb{T}^s,  \mbox{ where }\delta \mbox{ is a  constant, }
\end{equation}
then there exists a positive real number $\mathcal{M}<1$ such that 
\[
\Vert e_{-\gamma B}(t,T_{2i}) \Vert \leq \prod\limits_{s \in [T_{2i}, t)} \mathcal{M} \mbox{ for }t \in [T_{2i}, T_{2i+1}),\,\,i \in \mathbb{N}_0,
\]
\[
\Vert e_{-\gamma B}(t,T_{2i+1}) \Vert \leq \mathcal{M}^{\int_{T_{2i+1}}^{t}\vert \gamma(s) \vert ds } \mbox{ for }t \in [T_{2i+1}, T_{2i+2}), \,\,i \in \mathbb{N}_0,
\]
where $\Vert \cdot \Vert$ denotes the spectral norm of considered matrix at the point $t$.
\end{lemma}
\begin{proof}
Obviously, $\mathbb{T}^s \cup \mathbb{T}^d =\mathbb{T}$.
We consider two cases:
\begin{itemize}
\item[\textit{(i)}]
$t\in \mathbb{T}^s$; 
\item[\textit{(ii)}]
$t\in \mathbb{T}^d$. 
\end{itemize}
In case $(i)$, notice that since matrix $B$ is symmetric, then $I- \mu(t)\gamma(t) B$ is a symmetric matrix at the point $t$, too. Therefore $\Vert I- \mu(t)\gamma(t) B \Vert$ equals the maximum of the absolute value of eigenvalues of matrix  $I- \mu(t)\gamma(t) B$. It means
\[
\Vert e_{-\gamma B}(t,T_{2i}) \Vert = \prod\limits_{s \in (T_{2i}, t]} \!\!\! \Vert I- \mu(s)\gamma(s) B \Vert= \!\!\! \prod\limits_{s \in (T_{2i}, t]} (\max\limits_{i \in \{1,2, \dots, N\}}\{\vert 1- \mu(s)\gamma(s) \lambda_i \vert \})
\]
for $t \in [T_{2i}, T_{2i+1})$.
Because of positivity of $\mu$ on $\mathbb{T}^s$ and condition \eqref{ewa13_61}, we have $\vert \mu(s)\gamma(s) \lambda_i \vert = \mu(s)\gamma(s) \lambda_i $. Moreover, by \eqref{ewa13_7}, $\mu(s)\gamma(s) \lambda_i \in (0,1)$ for $i \in \{1,2, \dots, N\}$. We can conclude
\[
\Vert e_{-\gamma B}(t,T_{2i}) \Vert = \prod\limits_{s \in (T_{2i}, t]}(1-\min\limits_{i \in \{1,2, \dots, N\}}\{ \mu(s)\gamma(s) \lambda_i  \}).
\]
Again by \eqref{ewa13_7}, we have  
\[
-1< -\min\limits_{i \in \{1,2, \dots, N\}} \mu(s) \gamma(s) \lambda_i \leq -\delta  <0.
\]
From above 
\[
\Vert e_{-\gamma B}(t,T_{2i}) \Vert \leq \prod\limits_{s \in (T_{2i}, t]}(1-\delta)  = \prod\limits_{s \in (T_{2i}, t]} \mathcal{M}^*  = \prod\limits_{s \in [T_{2i}, t)} \mathcal{M}^* \mbox{ for }t \in [T_{2i}, T_{2i+1}),
\]
where $\mathcal{M}^* \colon = 1-\delta \in (0,1)$. \\
Case $(ii)$. Condition \eqref{ewa13_61} implies 

$(1^o)$ $\lambda_i>0$ for any $i=1,2,\dots, N$ and $\gamma(t)>0$ for any $t \in \mathbb{T}^d$ 

or 

$(2^o)$  $\lambda_i<0$ for any $i=1,2,\dots, N$ and $\gamma(t)<0$ for any $t \in \mathbb{T}^d$. \\
If $(1^o)$, then
\[
\Vert e_{-\gamma B}(t,T_{2i+1}) \Vert=\Vert e^{ B }\Vert^{-\int_{T_{2i+1}}^{t}\gamma(s)ds} = (\max\limits_{i \in \{1,2, \dots, N\}} \{e^{  \lambda_i }\})^{ - \int_{T_{2i+1}}^{t}\gamma(s)ds } 
\]
\[
= (\mathcal{M}^{**})^{\int_{T_{2i+1}}^{t}\gamma(s)ds }= (\mathcal{M}^{**})^{\int_{T_{2i+1}}^{t} \vert \gamma(s) \vert ds }
\]
for $t \in [T_{2i+1}, T_{2i+2})$, where $\mathcal{M}^{**}\colon = (\max\limits_{i \in \{1,2, \dots, N\}} \{e^{  \lambda_i }\})^{ - 1} \in(0,1)$.\\
If $(2^o)$, then
\[
\Vert e_{-\gamma B}(t,T_{2i+1}) \Vert=\Vert e^{ B }\Vert^{-\int_{T_{2i+1}}^{t}\gamma(s)ds} = (\max\limits_{i \in \{1,2, \dots, N\}} \{e^{  \lambda_i }\})^{ - \int_{T_{2i+1}}^{t}\gamma(s)ds } 
\]
\[
= (\mathcal{M}^{**})^{-\int_{T_{2i+1}}^{t}\gamma(s)ds }= (\mathcal{M}^{**})^{\int_{T_{2i+1}}^{t}\vert \gamma(s) \vert ds }
\]
for $t \in [T_{2i+1}, T_{2i+2})$, where $\mathcal{M}^{**}\colon = \max\limits_{i \in \{1,2, \dots, N\}} \{e^{  \lambda_i }\} \in(0,1)$.

Set $\mathcal{M}\colon= \max\{\mathcal{M}^{*}, \mathcal{M}^{**}\}$. Obviously $\mathcal{M}\in(0,1)$.\\
Hence $\Vert e_{-\gamma B}(t,T_{2i}) \Vert \leq  \prod\limits_{s \in [T_{2i}, t)} \mathcal{M}$ for $t \in [T_{2i}, T_{2i+1})$ \\and $\Vert e_{-\gamma B}(t,T_{2i+1}) \Vert \leq \mathcal{M}^{\int_{T_{2i+1}}^{t}\vert \gamma(s) \vert ds }$ for $t \in [T_{2i+1}, T_{2i+2})$.
\end{proof}
Next, we find the estimations of the norm of matrix $e_{-\gamma B}(t,T_0) $ in two cases: $t \in [T_{2i}, T_{2i+1})$ and $t \in [T_{2i+1}, T_{2i+2})$.
\begin{lemma}\label{L3}
If conditions \eqref{ewa13_61}-\eqref{ewa13_7} are satisfied,
then 
\begin{equation}\label{t1}
\Vert e_{-\gamma B}(t,T_{0}) \Vert \leq \big(\mathcal{M}^{\sum_{j=1}^{i}{\int_{T_{2j-1}}^{T_{2j}}}\vert \gamma(s) \vert ds}\big)  \big(\prod\limits_{s \in [T_{0}, t) \cap \mathbb{T}^s} \mathcal{M} \big) 
\end{equation}
 for $ t \in [T_{2i}, T_{2i+1})$, and
\begin{equation}\label{t2}
 \Vert e_{-\gamma B}(t,T_{0}) \Vert \leq \big( \prod\limits_{s \in [T_0,T_{2i+1}) \cap \mathbb{T}^s} \mathcal{M} \big)\big(\mathcal{M}^{\sum_{j=1}^{i}{\int_{T_{2j-1}}^{T_{2j}}\vert \gamma(s) \vert ds}+{\int_{T_{2i+1}}^{t}\vert \gamma(s) \vert ds}}\big)
\end{equation}
for $t \in [T_{2i+1}, T_{2i+2})$, where $i \in \mathbb{N}_0$.
\end{lemma} 
\begin{proof}
Let us rewrite function  $e_{-\gamma B}(t,T_{0})$ in the following form
\[
 e_{-\gamma B}(t,T_{0})= \big(\prod\limits_{s \in [T_0,T_1)}( I- \mu(s)\gamma(s) B )\big) \big( e^{ -B \int_{T_1}^{T_2}\gamma(s)ds} \big)
 \]
 \[
\cdot \big(\prod\limits_{s \in [T_2,T_3)}( I- \mu(s)\gamma(s) B )\big) \big( e^{ -B \int_{T_3}^{T_4}\gamma(s)ds} \big) 
 \]
 \[
 \cdots
 \]
 \[
\cdot \big(\prod\limits_{s \in [T_{2i-2},T_{2i-1})}( I- \mu(s)\gamma(s) B )\big) \big( e^{ -B \int_{T_{2i-1}}^{T_{2i}}\gamma(s)ds} \big) 
 \]
  \[
\cdot \big(\prod\limits_{s \in [T_{2i},t]}( I- \mu(s)\gamma(s) B )\big) \mbox{ for }t \in [T_{2i},T_{2i+1})
 \]
or
\[
 e_{-\gamma B}(t,T_{0})= \big(\prod\limits_{s \in [T_0,T_1)}( I- \mu(s)\gamma(s) B )\big) \big( e^{ -B \int_{T_1}^{T_2}\gamma(s)ds} \big)
 \]
 \[
\cdot \big(\prod\limits_{s \in [T_2,T_3)}( I- \mu(s)\gamma(s) B )\big) \big( e^{ -B \int_{T_3}^{T_4}\gamma(s)ds} \big) 
 \]
 \[
 \cdots
 \]
 \[
\cdot \big(\prod\limits_{s \in [T_{2i-2},T_{2i-1})}( I- \mu(s)\gamma(s) B )\big) \big( e^{ -B \int_{T_{2i-1}}^{T_{2i}}\gamma(s)ds} \big) 
 \]
  \[
\cdot \big(\prod\limits_{s \in [T_{2i},T_{2i+1)]}}( I- \mu(s)\gamma(s) B )\big) \big( e^{ -B \int_{T_{2i+1}}^{t}\gamma(s)ds} \big) \mbox{ for }t \in [T_{2i+1},T_{2i+2}).
 \]
 By submultiplicativity of the norm, for $t \in [T_{2i},T_{2i+1})$ we estimate the norm of matrix $  e_{-\gamma B}(t,T_{0})$
 \[
 \Vert e_{-\gamma B}(t,T_{0}) \Vert \leq \big(\prod\limits_{s \in [T_0,T_1)}\Vert I- \mu(s) \gamma(s) B \Vert \big) \big(\Vert e^{ B }\Vert^{-\int_{T_1}^{T_2} \vert \gamma(s) \vert ds}\big)
 \]
  \[
\cdot \big(\prod\limits_{s \in [T_2,T_3)}\Vert I- \mu(s)\gamma(s) B \Vert \big) \big(\Vert e^{ B }\Vert^{-\int_{T_3}^{T_4} \vert \gamma(s) \vert ds}\big)  
\]
\[
\cdots
\]
\[
 \cdot  \big(\prod\limits_{s \in [T_{2i-2},T_{2i-1})}\Vert I- \mu(s)\gamma(s) B \Vert \big) \big( \Vert e^{ B }\Vert^{-\int_{T_{2i-1}}^{T_{2i}} \vert \gamma(s) \vert ds} \big)  
\]
\[
\cdot \big(\prod\limits_{s \in [T_{2i},t)}\Vert I- \mu(s)\gamma(s) B \Vert \big) 
 \]
 \[
 \leq \big( \prod\limits_{s \in [T_0,T_1)} \mathcal{M} \big)   \big(\mathcal{M}^{\int_{T_1}^{T_2}\vert \gamma(s) \vert ds  }\big)   \big( \prod\limits_{s \in [T_2,T_3)} \mathcal{M} \big)    \big(\mathcal{M}^{\int_{T_3}^{T_4}\vert \gamma(s)\vert ds}\big) \dots
\]
\[
\cdot  \big(\mathcal{M}^{\int_{T_{2i-1}}^{T_{2i}}\vert \gamma(s)\vert ds}\big) \big( \prod\limits_{s \in [T_{2i},t)} \mathcal{M} \big)
 \]
 \[
 = \big( \prod\limits_{s \in [T_0,t) \cap \mathbb{T}^s} \mathcal{M} \big)\big(\mathcal{M}^{\int_{T_1}^{T_2}\vert \gamma(s)\vert ds + \int_{T_3}^{T_4}\vert \gamma(s)\vert ds  + \dots + \int_{T_{2i-1}}^{T_{2i}}\vert \gamma(s)\vert ds }\big)  
 \]
  \[
 = \big(\mathcal{M}^{\sum_{j=1}^{i}{\int_{T_{2j-1}}^{T_{2j}}\vert \gamma(s) \vert ds}} \big)  \big( \prod\limits_{s \in [T_0,t) \cap \mathbb{T}^s} \mathcal{M} \big) 
 \]
Analogously, for $t \in [T_{2i+1},T_{2i+2})$, we get
\[
 \Vert e_{-\gamma B}(t,T_{0}) \Vert \leq \big( \prod\limits_{s \in [T_0,T_{2i+1}) \cap \mathbb{T}^s} \mathcal{M} \big)\big(\mathcal{M}^{\sum_{j=1}^{i}{\int_{T_{2j-1}}^{T_{2j}}\vert \gamma(s) \vert ds}+{\int_{T_{2i+1}}^{t}\vert \gamma(s) \vert ds}}\big). 
 \]
\end{proof}

\begin{remark}\label{R1}
If conditions \eqref{ewa13_61}-\eqref{ewa13_7} are satisfied,
then 
\[
\Vert e_{-\gamma B}(t,T_{0}) \Vert \leq \big(\prod\limits_{s \in [T_{0}, t) \cap \mathbb{T}^s} \mathcal{M}  \big) \mbox{ for } t \in \mathbb{T}.
\]
\end{remark} 
\begin{proof}
Since $\mathcal{M} \in (0,1)$ and $\sum_{j=1}^{i}{\int_{T_{2j-1}}^{T_{2j}}\vert \gamma(s) \vert ds}+\int_{T_{2i+1}}^{t}\vert \gamma(s) \vert ds \geq 0$, thus 
\[
\mathcal{M}^{\sum_{j=1}^{i}{\int_{T_{2j-1}}^{T_{2j}}\vert \gamma(s) \vert ds}+{\int_{T_{2i+1}}^{t}\vert \gamma(s) \vert ds}} \leq 1 \mbox{  for }t \in \mathbb{T}.
\]
From the above, inequalities \eqref{t1} and \eqref{t2} imply
\[
\Vert e_{-\gamma B}(t,T_{0}) \Vert \leq \big(\prod\limits_{s \in [T_{0}, t) \cap \mathbb{T}^s} \mathcal{M} \big)   \mbox{ for } t \in [T_{2i}, T_{2i+1}),
\]
\[
 \Vert e_{-\gamma B}(t,T_{0}) \Vert \leq \big( \prod\limits_{s \in [T_0,T_{2i+1}) \cap \mathbb{T}^s} \mathcal{M} \big)   \mbox{ for } t \in [T_{2i+1}, T_{2i+2}),
\]
where $i \in \mathbb{N}_0$. This is our claim.
\end{proof}

The following result concerns of scalar case of exponential function on arbitrary time scale.
\begin{lemma}\label{L4}
Let $e_{\mathcal{L} \mathcal{M}^{-1}}( \cdot, T_0) \colon \mathbb{T} \to \mathbb{R}$. There hold
\[
 e_{\mathcal{L} \mathcal{M}^{-1}}( t, T_0) 
 = e^{\mathcal{L} \mathcal{M}^{-1}\sum_{j=1}^{i}  (T_{2j}-T_{2j-1})}\!\!\!\prod\limits_{s \in [T_{0}, t) \cap \mathbb{T}^s} \!\!\!(1+ \mu(s)\mathcal{L} \mathcal{M}^{-1})
\] 
 for $t \in [T_{2i}, T_{2i+1})$, and
\[
 e_{\mathcal{L} \mathcal{M}^{-1}}( t, T_0)  =  \!\!\!\!\!\!\!\!\! \prod\limits_{s \in [T_0,T_{2i+1}) \cap \mathbb{T}^s}  \!\!\!\!\!\!\!\!\!\!\!\! (1+ \mu(s)\mathcal{L} \mathcal{M}^{-1})\cdot e^{\mathcal{L} \mathcal{M}^{-1}(\sum_{j=1}^{i}  (T_{2j}-T_{2j-1})+(t-T_{2j+1}))}
\]
 for $t \in [T_{2i+1}, T_{2i+2})$,
where $i \in \mathbb{N}_0$.
\end{lemma} 
We are now in a position to present the main theorem of this paper.
\begin{theorem}\label{C2}
If conditions  \eqref{Lip}-\eqref{ewa13_7} are satisfied, and for any $t \in  \mathbb{T}$
\begin{equation}\label{C2e1}
\mbox{ there exists a positive constant }\mu^* \mbox{ such that } \mu(t) \leq \mu^*,
\end{equation}
\begin{equation}\label{e_11}
\lim\limits_{t \to \infty}  e^{\mathcal{L} \mathcal{M}^{-1}\sum_{j=1}^{i}  (T_{2j}-T_{2j-1})}\mathcal{M}^{\sum_{j=1}^{i}{\int_{T_{2j-1}}^{T_{2j}}}\vert \gamma(s) \vert ds}\!\!\!\prod\limits_{s \in [T_{0}, t) \cap \mathbb{T}^s} \!\!\!(\mathcal{M}+ \mu^*\mathcal{L})=0
\end{equation}
and
\begin{align}\label{e_21}
\lim\limits_{t \to \infty} & e^{\mathcal{L} \mathcal{M}^{-1}(\sum_{j=1}^{i}  (T_{2j}-T_{2j-1})+(t-T_{2j+1}))}\\ \nonumber
& \cdot \mathcal{M}^{\sum_{j=1}^{i}{\int_{T_{2j-1}}^{T_{2j}}\vert \gamma(s) \vert ds}+{\int_{T_{2i+1}}^{t}\vert \gamma(s) \vert ds}}\!\!\!\prod\limits_{s \in [T_{0}, T_{2i+1}) \cap \mathbb{T}^s} \!\!\!(\mathcal{M}+ \mu^*\mathcal{L})=0,
\end{align}
then equation \eqref{eq5} is exponentially stable.
\end{theorem}
\begin{proof}
Taking the norm of the both sides of equation \eqref{e2}, we obtain 
\[
\Vert \varepsilon(t) \Vert= \Vert e_{-\gamma B}(t, T_0)\varepsilon_{T_0}
+\int\limits_{T_0}^{t}e_{-\gamma B}(t, \sigma(\tau))\big(F(\tau,x_0(\tau)\mathds{1})-F(\tau,x(\tau))\big) \Delta \tau \Vert.
\]
Using properties of the norm, we get
\[
\Vert \varepsilon(t) \Vert \leq \Vert e_{-\gamma B}(t, T_0)\Vert \Vert\varepsilon_{T_0}\Vert
+ \Vert\int\limits_{T_0}^{t}e_{-\gamma B}(t, \sigma(\tau))\big(F(\tau,x_0(\tau)\mathds{1})-F(\tau,x(\tau))\big) \Delta \tau \Vert,
\]
and consequently
\[
\Vert \varepsilon(t) \Vert \leq \Vert\varepsilon_{T_0}\Vert \Vert e_{-\gamma B}(t, T_0)\Vert 
+ \int\limits_{T_0}^{t} \Vert e_{-\gamma B}(t, \sigma(\tau)) \Vert \Vert \big(F(\tau,x_0(\tau)\mathds{1})-F(\tau,x(\tau))\big)\Vert \Delta \tau.
\]
By condition \eqref{Lip}, we obtain
\[
\Vert \varepsilon(t) \Vert \leq \Vert\varepsilon_{T_0}\Vert \Vert e_{-\gamma B}(t, T_0)\Vert 
+ \mathcal{L} \int\limits_{T_0}^{t} \Vert e_{-\gamma B}(t, \sigma(\tau)) \Vert \Vert \varepsilon(\tau) \Vert \Delta \tau.
\]
For $t \in [T_{2i}, T_{2i+1})$, using \eqref{t1}, we estimate
\[
\Vert \varepsilon(t) \Vert \leq \Vert\varepsilon_{T_0}\Vert \big(\mathcal{M}^{\sum_{j=1}^{i}{\int_{T_{2j-1}}^{T_{2j}}\vert \gamma(s) \vert ds}}\big) \big(\prod\limits_{s \in [T_{0}, t) \cap \mathbb{T}^s} \mathcal{M} \big) 
\]
\[
+ \mathcal{L} \int\limits_{T_0}^{t} \big(\mathcal{M}^{\int_{\sigma(\tau)}^{T_{2i}}\vert \gamma(s) \vert ds}\big)\big(\prod\limits_{s \in [\sigma(\tau), t) \cap \mathbb{T}^s} \mathcal{M} \big)  \Vert \varepsilon(\tau) \Vert \Delta \tau.
\]
Multiplying the both sides of the above inequality by 
\[
\big(\mathcal{M}^{-\sum_{j=1}^{i}{\int_{T_{2j-1}}^{T_{2j}}\vert \gamma(s) \vert ds}}\big)\big(\prod\limits_{s \in (T_{0}, t) \cap \mathbb{T}^s} \mathcal{M}^{-1} \big),
\]
we obtain
\[
\Vert \varepsilon(t) \Vert \big(\mathcal{M}^{-\sum_{j=1}^{i}{\int_{T_{2j-1}}^{T_{2j}}\vert \gamma(s) \vert ds}}\big)\big(\prod\limits_{s \in (T_{0}, t) \cap \mathbb{T}^s} \mathcal{M}^{-1} \big)
\]
\[
\leq  \Vert\varepsilon_{T_0}\Vert  \mathcal{M} +\int\limits_{T_0}^{t}\mathcal{L}\Vert \varepsilon(\tau) \Vert \big(\mathcal{M}^{-(\sum_{j=1}^{i}{\int_{T_{2j-1}}^{T_{2j}}\vert \gamma(s) \vert ds + \int_{T_{2i}}^{\sigma(\tau)}\vert \gamma(s) \vert ds})}\big)\big(\prod\limits_{s \in (T_0, \tau]\cap \mathbb{T}^s}{\mathcal{M}^{-1}} \big) \Delta \tau.
\]
Since $\mathcal{M}^{-\int_{T_0}^{T_0}\vert \gamma(s) \vert ds}=1$,
\[
\Vert \varepsilon(t) \Vert \big(\mathcal{M}^{-\sum_{j=1}^{i}{\int_{T_{2j-1}}^{T_{2j}}\vert \gamma(s) \vert ds}}\big)\big(\prod\limits_{s \in (T_{0}, t) \cap \mathbb{T}^s} \mathcal{M}^{-1} \big)
\]
\[
\leq  \Vert\varepsilon_{T_0}\Vert  \mathcal{M}^{-\int_{T_0}^{T_0}\vert \gamma(s) \vert ds} \cdot \mathcal{M} 
\]
\[
+\int\limits_{T_0}^{t}\mathcal{L}\Vert \varepsilon(\tau) \Vert \big(\mathcal{M}^{-(\sum_{j=1}^{i}{\int_{T_{2j-1}}^{T_{2j}}\vert \gamma(s) \vert ds + \int_{T_{2i}}^{\sigma(\tau)}\vert \gamma(s) \vert ds})}\big)\big(\prod\limits_{s \in (T_0, \tau]\cap \mathbb{T}^s}{\mathcal{M}^{-1}} \big) \Delta \tau.
\] 
Using $\sigma(\tau)= \tau$, we get
\[
\Vert \varepsilon(t) \Vert \big(\mathcal{M}^{-\sum_{j=1}^{i}{\int_{T_{2j-1}}^{T_{2j}}\vert \gamma(s) \vert ds}}\big)\big(\prod\limits_{s \in (T_{0}, t) \cap \mathbb{T}^s} \mathcal{M}^{-1} \big)
\]
\[
\leq  \Vert\varepsilon_{T_0}\Vert  \mathcal{M}^{-\int_{T_0}^{T_0}\vert \gamma(s) \vert ds} \cdot \mathcal{M} 
\]
\[
+\int\limits_{T_0}^{t}\mathcal{L}\Vert \varepsilon(\tau) \Vert \big(\mathcal{M}^{-(\sum_{j=1}^{i}{\int_{T_{2j-1}}^{T_{2j}}\vert \gamma(s) \vert ds + \int_{T_{2i}}^{\tau}\vert \gamma(s) \vert ds})}\big)\big(\prod\limits_{s \in (T_0, \tau]\cap \mathbb{T}^s}{\mathcal{M}^{-1}} \big) \Delta \tau.
\] 
By Gr\"{o}nwall Lemma (see \cite[p.~257]{bookBohner2001}), it leads to inequality 
\[
\Vert \varepsilon(t) \Vert \big(\mathcal{M}^{-\sum_{j=1}^{i}{\int_{T_{2j-1}}^{T_{2j}}\vert \gamma(s) \vert ds}}\big)\big(\prod\limits_{s \in (T_{0}, t) \cap \mathbb{T}^s} \mathcal{M}^{-1} \big) \leq \Vert \varepsilon_{T_0} \Vert \mathcal{M} e_{\mathcal{L}\mathcal{M}^{-1}}(t, T_0).
\]
Using Lemma \ref{L4}
\[
\Vert \varepsilon(t) \Vert \big(\mathcal{M}^{-\sum_{j=1}^{i}{\int_{T_{2j-1}}^{T_{2j}}\vert \gamma(s) \vert ds}}\big)\big(\prod\limits_{s \in (T_{0}, t) \cap \mathbb{T}^s} \mathcal{M}^{-1} \big)
\]
\[
\leq \Vert \varepsilon_{T_0} \Vert \mathcal{M} \big(e^{\mathcal{L} \mathcal{M}^{-1}\sum_{j=1}^{i}  (T_{2j}-T_{2j-1})}\big)\big(\!\!\!\prod\limits_{s \in [T_{0}, t) \cap \mathbb{T}^s} \!\!\!(1+ \mu(s)\mathcal{L} \mathcal{M}^{-1})\big).
\]
Hence
\[
\Vert \varepsilon(t) \Vert \leq \Vert \varepsilon_{T_0} \Vert  \big(e^{\mathcal{L} \mathcal{M}^{-1}\sum_{j=1}^{i}  (T_{2j}-T_{2j-1})}\big)
\]
\[
\cdot \big(\mathcal{M}^{\sum_{j=1}^{i}{\int_{T_{2j-1}}^{T_{2j}}\vert \gamma(s) \vert ds}}\big)\big(\!\!\!\prod\limits_{s \in [T_{0}, t) \cap \mathbb{T}^s} \!\!\!(\mathcal{M}+ \mu(s)\mathcal{L})\big).
\]
By \eqref{C2e1},
\begin{align}\label{e_1}
\Vert \varepsilon(t) \Vert &\leq \Vert \varepsilon_{T_0} \Vert  \big(e^{\mathcal{L} \mathcal{M}^{-1}\sum_{j=1}^{i}  (T_{2j}-T_{2j-1})}\big)\\ \nonumber
&\cdot \big(\mathcal{M}^{\sum_{j=1}^{i}{\int_{T_{2j-1}}^{T_{2j}}\vert \gamma(s) \vert ds}}\big)\big(\!\!\!\prod\limits_{s \in [T_{0}, t) \cap \mathbb{T}^s} \!\!\!(\mathcal{M}+ \mu^*\mathcal{L})\big).
\end{align}
Analogously, for $t \in (T_{2i+1}, T_{2i+2}] $
\begin{align}\label{e_2}
\Vert \varepsilon(t) \Vert & \leq \Vert \varepsilon_{T_0} \Vert  \big(e^{\mathcal{L} \mathcal{M}^{-1}\sum_{j=1}^{i}  (T_{2j}-T_{2j-1})} \big) \\ \nonumber
& \cdot \big(\mathcal{M}^{\sum_{j=1}^{i}{\int_{T_{2j-1}}^{T_{2j}}\vert \gamma(s) \vert ds}+{\int_{T_{2i+1}}^{t}\vert \gamma(s) \vert ds}}\big)\big(\!\!\!\prod\limits_{s \in [T_{0}, T_{2i+1}) \cap \mathbb{T}^s} \!\!\!(\mathcal{M}+ \mu^*\mathcal{L}\big).
\end{align}
Set
\[
\textup{sum}(i) \colon = \sum_{j=1}^{i} \Big({\mathcal{L} \mathcal{M}^{-1} (T_{2j}-T_{2j-1}) + \textup{ln} \mathcal{M}  {{\int_{T_{2j-1}}^{T_{2j}}\vert \gamma(s) \vert ds}} \Big)},
\]
\[
e^{*}_d(t, T_0) \colon =e^{\textup{sum}(i)}\!\!\!\prod\limits_{s \in [T_{0}, t) \cap \mathbb{T}^s} \!\!\!(\mathcal{M}+ \mu^*\mathcal{L}) \mbox{ for }t\in (T_{2i}, T_{2i+1}],
\]
\[
e^{**}_d(t, T_0) \colon =e^{\textup{sum}(i)+{\int_{T_{2i+1}}^{t}\vert \gamma(s) \vert ds}}\!\!\!\prod\limits_{s \in [T_{0}, T_{2i+1}) \cap \mathbb{T}^s} \!\!\!(\mathcal{M}+ \mu^*\mathcal{L})\mbox{ for }t\in (T_{2i+1}, T_{2i+2}]
\]
for $i \in \mathbb{N}_0$, and
\[
e_d(t, T_0) \colon =
\begin{cases}
e^{*}_d(t, T_0) &\mbox{ for }t\in (T_{2i}, T_{2i+1}], \\
e^{**}_d(t, T_0) &\mbox{ for }t\in (T_{2i+1}, T_{2i+2}].
\end{cases}
\]
By \eqref{e_11} and \eqref{e_21}, inequalities \eqref{e_1} and \eqref{e_2} imply the thesis.
\end{proof}
\begin{corollary}\label{C1}
If conditions  \eqref{Lip}-\eqref{ewa13_7} and \eqref{C2e1} are satisfied, 
\begin{align}\label{e100}
&\mbox{ for any }t \in \mathbb{T}^s \mbox{ there exists } \tilde{t} \in \mathbb{T}^d \mbox{ such that } \tilde{t}>t \mbox{ and }\\ \nonumber
&\mbox{ for any }t \in \mathbb{T}^d \mbox{ there exists } \tilde{t} \in \mathbb{T}^s \mbox{ such that } \tilde{t}>t,
\end{align}
\begin{equation}\label{e1}
\mathcal{M}+ \mu^* \mathcal{L}  < 1,
\end{equation}
and
\begin{equation}\label{e_7}
\lim\limits_{i \to \infty} e^{\textup{sum}(i)} < \infty, 
\end{equation}
then equation \eqref{eq5} is exponentially stable.
\end{corollary}
\begin{proof}
By \eqref{e100} we get, $t \to \infty$ iff $i \to \infty$. Since $0<\mathcal{M}+ \mu^*\mathcal{L} <1$ and $\mathcal{M} \in (0,1)$, by properties of functions $\mathcal{M}^t$ and $e^t$, condition \eqref{e_7} implies conditions  \eqref{e_11} and \eqref{e_21}.  Hence assumptions of Theorem~\ref{C1} are satisfied. So, the thesis holds.
\end{proof}
\begin{example}\label{Ex1}
Let
\[
\mathbb{T}=\bigcup\limits_{i=3}^{\infty} \left[\frac{i}{2}, \frac{i}{2}+\frac{1}{i^3}\right].
\]
Here $\mathbb{T}^d=\bigcup\limits_{i=3}^{\infty} \left[\frac{i}{2}, \frac{i}{2}+\frac{1}{i^3}\right)$ and $\mathbb{T}^s=\{\frac{i}{2}+\frac{1}{i^3} \colon i \in \mathbb{N}, i \geq 3\}$,
\[
T_0=T_1=1.500, \,\, T_2\approx 1.537, \,\, T_3 = 2.000, \cdots,\mu^*=0.500. 
\]
Moreover, let
\[
F(t,x)= 0.100\frac{x}{t^2}, \,\, \gamma(t)\equiv 0.500(t-1) 
\]
and  
\begin{equation}\label{M}
B= 
	\left( {\begin{array}{rrrr}
   2 & 0 & -1 & -1\\
	0 & 3 & 0 & 0 \\
  -1 & 0 & 3 & -1\\
	-1 & 0 & -1 & 3\\
	\end{array} } \right)
 \end{equation}
in equation \eqref{eq5}. 
There is $\mathcal{L}= 0.100$, $\lambda_1= 2 - \sqrt{2}$, $\lambda_2=2$, $\lambda_3= 3$ and $\lambda_4= 2 + \sqrt{2}$. It follows from these that $\lambda_{\min}=\min\{\lambda_1, \lambda_2,\lambda_3, \lambda_4\}\approx 0.585$ and
\[
\max\{ 1-0.462 \cdot 0.250 \cdot 0.585, e^{-0.585}\} \leq \max\{0.933, 0.557\} = 0.933 = \colon \mathcal{M}.
\]
From the above
\[
\prod\limits_{s \in \mathbb{T}^s} (\mathcal{M}+ \mu^*\mathcal{L})\approx\prod\limits_{s \in \mathbb{T}^s}0.983=0
\]
and
\[
\lim\limits_{i \to \infty}e^{\textup{sum}(i)}  \approx    \lim\limits_{i \to \infty}e^{ {\sum_{j=3}^{i}(0.108j^{-3}-0.070 \cdot 0.250 (i^{-2}-2i^{-3}+i^{-6}))}} < \infty.
\]
All assumptions of Corollary \ref{C1} are satisfied, thus equation \eqref{eq5} is exponentially stable. System \eqref{eq5} achieves consensus exponentially.
\end{example}
In Example \ref{Ex1} there is
\[
\lim\limits_{i \to \infty}  e^{\mathcal{L} \mathcal{M}^{-1}\sum_{j=1}^{i}  (T_{2j}-T_{2j-1})} \approx    \lim\limits_{i \to \infty}e^{ 0.108{\sum_{j=3}^{i}j^{-3}}}<\infty,
\] 
but this condition is not required for exponential stability of \eqref{eq5} (see Example~\ref{Ex5}).
\begin{remark}\label{R41}
If conditions  \eqref{Lip}-\eqref{ewa13_7}, \eqref{C2e1} and \eqref{e1} are satisfied, and
\begin{equation}\label{es71}
\gamma(t)\equiv \gamma \in \mathbb{R},
\end{equation}
and
\begin{equation}\label{es72}
\mathcal{L} \mathcal{M}^{-1} + \gamma\textup{ln} \mathcal{M} <0,
\end{equation}
then equation \eqref{eq5} is exponentially stable.
\end{remark}
\begin{proof}
If condition \eqref{es71} holds, then
\[
\textup{sum}(i) = \sum_{j=1}^{i} \Big({\mathcal{L} \mathcal{M}^{-1} (T_{2j}-T_{2j-1}) + \gamma(T_{2j}-T_{2j-1})\textup{ln} \mathcal{M}   \Big)}
\]
\[
 = ({\mathcal{L} \mathcal{M}^{-1}  + \gamma \,\,\textup{ln} \mathcal{M}}) \sum_{j=1}^{i} (T_{2j}-T_{2j-1}).
\]
By \eqref{es72}, we see that $\textup{sum}(i)< 0$ for any $i \in \mathbb{N}$, and $e^{\textup{sum}(i)}$ is a positive, decreasing function of variable $i\in \mathbb{N}$. Here $\prod\limits_{s \in \mathbb{T}^s}(\mathcal{M}+ \mu^* \mathcal{L})$ as well as $e^{\textup{sum}(i)}$ for any $i \in \mathbb{N}$, are bounded. If the cardinality of set $\mathbb{T}^s$ is infinity, then $\lim\limits_{s \to \infty}\prod\limits_{s \in \mathbb{T}^s}(\mathcal{M}+ \mu^* \mathcal{L})=0$. If the cardinality of set $\mathbb{T}^d$ is infinity, then $\lim\limits_{i \to \infty}e^{\textup{sum}(i)}=0$. Thus, by Theorem \ref{C2}, we obtain the thesis.
\end{proof}
\begin{example}\label{Ex5}
Let
\[
\mathbb{T}=\bigcup\limits_{i=3}^{\infty} \left[\frac{i}{2}, \frac{i}{2}+\frac{1}{i}\right].
\]
Here $\mathbb{T}^d=\bigcup\limits_{i=3}^{\infty} \left[\frac{i}{2}, \frac{i}{2}+\frac{1}{i}\right)$ and $\mathbb{T}^s=\{\frac{i}{2}+\frac{1}{i} \colon i \in \mathbb{N}, i \geq 3\}$,
\[
T_0=T_1=1.500, \,\, T_2\approx 1.833, \,\, T_3= 2.000, \cdots, \,\, 
\]
\[
\mu(t)=\frac{1}{2}-\frac{t}{2}+\frac{1}{2}\sqrt{t^2-2}  \mbox{ for } t \in \mathbb{T}^s, \,\, \mu^*=0.500.
\]
Moreover, let
\[
F(t,x)= \frac{0.250}{t^2} \left(\sin x_1, \sin x_2, \sin x_3, \sin x_4 \right), \,\, \gamma(t)\equiv 2.000,
\]
and matrix $B$ is given by \eqref{M} in equation \eqref{eq5}. 
There is $\mathcal{L}= 0.250$, $\lambda_{\min} \approx 0.585$, 
\[
\max\{ 1-0.333 \cdot 2.000 \cdot 0.585, e^{-0.585}\} \leq \max\{0.390, 0.557\} = 0.557 = \colon \mathcal{M}.
\]
Finally,
\[
\mathcal{L} \mathcal{M}^{-1} + \gamma\textup{ln} \mathcal{M} \approx 0.449-1.170=-0.721 <0.
\] 
All assumptions of Remark \ref{R41} hold, thus equation \eqref{eq5} is exponentially stable. 
\end{example}
In Example \ref{Ex5} there is
\[
\lim\limits_{i \to \infty}  e^{\mathcal{L} \mathcal{M}^{-1}\sum_{j=1}^{i}  (T_{2j}-T_{2j-1})}\approx \lim\limits_{i \to \infty} e^{0.449\sum_{j=1}^{i} \frac{1}{i}}=\infty,
\]
even that system \eqref{eq5} achieves consensus exponentially.
\begin{corollary}\label{R4}
If conditions  \eqref{Lip}-\eqref{ewa13_7} and \eqref{e1} are satisfied, and
\begin{equation}\label{es7}
\sum\limits_{i=0}^{\infty}(T_{2i+2}-T_{2i+1}) < \infty,
\end{equation}
then equation \eqref{eq5} is exponentially stable.
\end{corollary}
\begin{proof}
Since \eqref{es7} holds, 
\[
e^{\textup{sum}(i)}=constant.
\]
Hence, reminding that cardinality of the set $\mathbb{T}^s$ is infinity, by \eqref{e1}, we obtain
\[
\lim\limits_{i \to \infty} e^{\textup{sum}(i)}\!\!\!\prod\limits_{s \in [T_{0}, T_{2i}) \cap \mathbb{T}^s} \!\!\!(\mathcal{M}+ \mu^*\mathcal{L})
\]
\[
=\lim\limits_{i \to \infty} c^* \big(\!\!\!\prod\limits_{s \in [T_{0}, T_{2i}) \cap \mathbb{T}^s} \!\!\!(\mathcal{M}+ \mu^*\mathcal{L})\big) 
\]
\[
= c^*\prod\limits_{s \in \mathbb{T}^s} (\mathcal{M}+ \mu^*\mathcal{L})=0,
\]
where $c^*=e^{\textup{sum}(i)}$.
\end{proof}
For two possible cases of carrying out of assumption \eqref{es7} see Example \ref{Ex1} and Example \ref{Ex6}.

Theorem \ref{C2} generalize Theorem 2 \cite{osz2019}. The following example presents an equation on time scale for which Theorem 2 \cite{osz2019} can not be applicable, but our Corollary \ref{R4} of Theorem \ref{C2} can be. 
\begin{example}\label{Ex2}
Let
\[
\mathbb{T}=\{i \colon i \in \mathbb{N}\} \cup \{i+\frac{1}{j+1} \colon i,j \in \mathbb{N}, j\geq2\}.
\]
Here $\mathbb{T}^d=\{i \colon i \in \mathbb{N}\}$ and $\mathbb{T}^s=\{i+\frac{1}{j+1} \colon i,j \in \mathbb{N}, j\geq2\}$,
\[
T_0=1, \,\, T_1=\infty, \,\, \mu(t)=\frac{(t-i)^2}{1+t-i} \mbox{ for } t \in \mathbb{T}^s, \,\, \mu^*=0.500.
\]
Set $F(t,x)=0.250 x$, 
\[
\gamma(t)=
\begin{cases}
\frac{1}{\mu(t)} &\mbox{ for }t \in \mathbb{T}^s\\
0 &\mbox{ for }t \in \mathbb{T}^d,
\end{cases}
\]
and $B$ is given by \eqref{M} in equation \eqref{eq5}. 
There is $\mathcal{L}= 0.250$, $\lambda_{\min}= 0.585$, 
\[
\max\{ 1-1 \cdot 0.585, e^{-0.585}\} \leq \max\{0.515, 0.557\} = 0.557 = \colon \mathcal{M}.
\]
Hence
\[
\prod\limits_{s \in \mathbb{T}^s} (\mathcal{M}+ \mu^*\mathcal{L})\approx\prod\limits_{s \in \mathbb{T}^s} 0.682=0.
\] 
All assumptions of Corollary \ref{R4} are satisfied, thus equation \eqref{eq5} is exponentially stable. 
\end{example}
Since $\liminf_{t \to \infty}\mu(t) = 0$ results obtained in \cite{osz2019} can not be applied.

The following examples show two different situations concerning time scale in which condition \eqref{es7} is satisfied. In the first example, $\mathbb{T}^d$ is a bounded set. In the second one, set $\mathbb{T}^d$ is unbounded.
\begin{example}\label{Ex6}
Let
\[
\mathbb{T}=[1,2] \cup [3,7] \cup \{n \colon n \in \mathbb{N}, n\geq 8 \}.
\]
Here $\mathbb{T}^d=[1,2] \cup [3,7]$ is bounded and $\mathbb{T}^s=\{n \colon n \in \mathbb{N}, n\geq 8 \}$. We see that
\[
T_0=1, \,\, T_1=T_0=1, \,\, T_2=2, \,\, T_3=3, \,\, T_4=7, \,\, T_5=8, \,\,T_6=\infty,
\]
\[
\mu(t)=1 \mbox{ for } t \in \mathbb{T}^s, \,\, \mu^*=1.
\]
Let also
\[
F(t,x)=\frac{1}{4\sqrt{t}} \left(\sin x_1, \sin x_2, \sin x_3, \sin x_4 \right), \,\, \gamma(t)=  2+\cos t,
\]
and matrix $B$ is given by  \eqref{M} in equation \eqref{eq5}. There is $\mathcal{L}= 0.250$, $\lambda_{\min}= 0.585$, 
\[
\mathcal{M}\colon = \max\{ 1-1 \cdot 1 \cdot 0.585, e^{-0.585}\}\approx 0.557 <1.
\]
In consequence
\[
\prod\limits_{s \in \mathbb{T}^s} (\mathcal{M}+ \mu^*\mathcal{L})\approx \prod\limits_{s \in \mathbb{T}^s} 0.807=0.
\] 
All assumptions of Corollary \ref{R4} are satisfied, thus equation \eqref{eq5} is exponentially stable. 
\end{example}

\begin{example}\label{Ex8}
Let
\[
\mathbb{T}=\bigcup\limits_{i=3}^{\infty} [\frac{i}{2}+\frac{1}{i+1}, \frac{i}{2}+\frac{1}{i}].
\]
Here either $\mathbb{T}^d = \bigcup\limits_{i=3}^{\infty} [\frac{i}{2}+\frac{1}{i+1}, \frac{i}{2}+\frac{1}{i})$ or $\mathbb{T}^s=\{ \frac{i}{2}+\frac{1}{i} \colon i \in \mathbb{N}, i \geq 3 \}$ are unbounded sets. We see that
\[
T_0=T_1=1.750, \,\, T_2\approx 1.833, \,\, T_3=2.200, \,\, T_4=2.250,\cdots,  
\]
\[
\mu(t)=\frac{1}{2} - \frac{2}{(t+\sqrt{t^2-2})(2+t+\sqrt{t^2-2})} \mbox{ for } t \in \mathbb{T}^s, \,\, \mu^*=0.500.
\]
Moreover
\[
F(t,x)=\frac{x}{4 t}, \,\, \gamma(t)=\frac{1}{4}t^2,
\]
and matrix $B$ is given by  \eqref{M} in equation \eqref{eq5}. There is $\mathcal{L}= 0.250$, $\lambda_{\min}= 0.585$, 
\[
\max\{ 1-0.366 \cdot 0.765 \cdot 0.585, e^{-0.585}\} < 0.836 = \colon \mathcal{M}.
\]
Hence
\[
\prod\limits_{s \in \mathbb{T}^s} (\mathcal{M}+ \mu^*\mathcal{L})\approx\prod\limits_{s \in \mathbb{T}^s} 0.961=0.
\] 
All assumptions of Corollary \ref{R4} are satisfied, thus equation \eqref{eq5} is exponentially stable. 
\end{example}
Notice that in Example \ref{Ex8} there is
\[
\lim\limits_{i \to \infty}  e^{\mathcal{L} \mathcal{M}^{-1}\sum_{j=1}^{i}  (T_{2j}-T_{2j-1})}= \lim\limits_{i \to \infty} e^{0.299\sum_{j=1}^{i} \frac{1}{i(i+1)}}=e^{0.299}<\infty.
\]

\begin{remark}\label{R5}
If conditions  \eqref{Lip}-\eqref{ewa13_7} are satisfied, 
\[
\sum_{j=1}^{\infty} {\int_{T_{2j-1}}^{T_{2j}}\vert \gamma(s) \vert ds} < \infty,
\]
and
\[
\lim\limits_{i \to \infty} e^{\mathcal{L} \mathcal{M}^{-1}\sum_{j=1}^{i}  (T_{2j}-T_{2j-1})}\cdot \!\!\! \prod\limits_{s \in [T_{0}, T_{2i}) \cap \mathbb{T}^s} \!\!\!(\mathcal{M}+ \mu^*\mathcal{L})=0,
\]
then equation \eqref{eq5} is exponentially stable.
\end{remark}
(See Example \ref{Ex1})
\begin{remark}\label{R6}
Let conditions  \eqref{Lip}-\eqref{ewa13_7} be satisfied. If the cardinality of the set $\mathbb{T}^s$ is finite and  
$sum(i)<0$ for any $i \in \mathbb{N}$ then equation \eqref{eq5} is exponentially stable.
\end{remark}

\begin{example}\label{Ex9}
Let
\[
\mathbb{T}=\{1\} \cup \{11\} \cup [12, \infty).
\]
Here $\mathbb{T}^d=[12, \infty)$ is unbounded set whereas $\mathbb{T}^s=\{1\} \cup \{11\}$ is bounded, and 
\[
T_0=1, \,\, T_1=12, \,\,  T_2=\infty, \,\, \mu(1)=10, \,\, \mu(11)=1, \,\, \mu^*=10.
\]
Let
\[
F(t,x)=0.1 x, \,\, \gamma(t)=1,
\]
and matrix $B$ is given by  \eqref{M} in equation \eqref{eq5}. There is $\mathcal{L}= 0.100$, $\lambda_{\min} \approx 0.585$, 
\[
\max\{ 1- 0.585, e^{-0.585}\} < 0.557 = \colon \mathcal{M}.
\]
Hence
\[
\textup{sum}(i) \approx 0.180 (T_{2}-T_1) -0.585 {{\int_{T_1}^{T_2} ds}}=-0.405 (T_{2}-T_{1}) = -\infty <0.
\] 
All assumptions of Remark \ref{R6} are satisfied, thus equation \eqref{eq5} is exponentially stable. 
\end{example}
Notice that in Example \ref{Ex9} condition \eqref{e1} does not hold.
\section*{Acknowledgments}
The second author was supported by the Polish National Science Center grant on the basis
of decision DEC-2014/15/B/ST7/05270.


\end{document}